\documentclass{amsart}
\usepackage[T1]{fontenc}
\usepackage[latin9]{inputenc}
\usepackage{graphicx}
\usepackage{amssymb}
\usepackage{amsmath}

\makeatletter
\newenvironment{lyxlist}[1]
{\begin{list}{}
{\settowidth{\labelwidth}{#1}
 \setlength{\leftmargin}{\labelwidth}
 \addtolength{\leftmargin}{\labelsep}
 }}
{\end{list}}
\theoremstyle{definition}
    \newtheorem{thm}{Theorem}
    \newtheorem{prop}[thm]{Proposition}
    \newtheorem{lem}[thm]{Lemma}
    \newtheorem{cor}[thm]{Corollary}
    \newtheorem{dfn}[thm]{Definition}
    \newtheorem{example}[thm]{Example}
\theoremstyle{remark}
    \newtheorem*{notation*}{Notation}
    \newtheorem*{remark*}{Remark}
\theoremstyle{definition}
    \newtheorem*{example*}{Example}

\newcommand\ordering{\ensuremath{\prec}}
\newcommand\ordlt{\ensuremath{\prec}}
\newcommand\ordgt{\ensuremath{\succ}}
\newcommand\ordle{\ensuremath{\preceq}}

\newcommand\polyring{\ensuremath{\mathcal{R}}}

\newcommand\SPoly[1]{\ensuremath{S_{\ordlt}\left({#1}\right)}}

\newcommand\lcm{\ensuremath{\mathrm{lcm}}}

\newcommand\lt[1]{\ensuremath{\mathrm{lt}_{\ordlt}\left(#1\right)}}
\newcommand\lc[1]{\ensuremath{\mathrm{lc}_{\ordlt}\left(#1\right)}}

\newcommand\closer{\ensuremath{\ \Diamond}}

\theoremstyle{plain}
\newtheorem*{maintheorem}{Main Theorem}

\makeatother

\begin{document}

\title[An Extension of Buchberger's Criteria]{An Extension of Buchberger's Criteria for Gr\"obner basis decision}

\author{John Perry}

\email{john.perry@usm.edu}

\urladdr{www.math.usm.edu/perry}

\address{University of Southern Mississippi\\
Department of Mathematics, Box 5045\\
Hattiesburg, MS 39406 USA}



\thanks{Part of this work was conducted during the Special Semester on Gr\"obner
bases, February 1--July 31, 2006, organized by RICAM, Austrian Academy
of Sciences, and RISC, Johannes Kepler University, Linz, Austria.}

\begin{abstract}
Two fundamental questions in the theory of Gr\"obner bases are decision
(``Is a basis $G$ of a polynomial ideal a Gr\"obner basis?'') and
transformation (``If it is not, how do we transform it into a Gr\"obner
basis?'') This paper considers the first question. It is well-known
that $G$ is a Gr\"obner basis if and only if a certain set of polynomials
(the $S$-poly\-nomials) satisfy a certain property. In general there
are $m\left(m-1\right)/2$ of these, where $m$ is the number of polynomials
in $G$, but criteria due to Buchberger and others often allow one
to consider a smaller number.

This paper presents two original results. The first is a new characterization
theorem for Gr\"obner bases that makes use of a new criterion that
extends Buchberger's Criteria. The second is the identification of
a class of polynomial systems $G$ for which the new criterion has
dramatic impact, reducing the worst-case scenario from $m\left(m-1\right)/2$
$S$-poly\-nomials to $m-1$.
\end{abstract}
\maketitle

\section{\label{sec:Introduction}Introduction}

Gr\"obner bases ease significantly the investigation of many important
questions in commutative algebra and algebraic geometry. Fundamental
questions in the theory of Gr\"obner bases include (1) the decision
problem, \emph{Is a basis $G$ of a polynomial ideal a Gr\"obner
basis?} and (2) the transformation problem, \emph{If it is not, how
do we transform it into one?} This paper considers question (1).

Buchberger \cite{Buchberger65} showed that $G$ is a Gr\"obner basis
if and only if the $S$-poly\-nomial of every pair of the polynomials
in $G$ satisfies a certain property. Ordinarily, if $G$ contains
$m$ polynomials, one has to examine $m\left(m-1\right)/2$ $S$-poly\-nomials.
Buchberger and others
\cite{Buchberger65,KollreiderBuchberger76,Buchberger79,GM88,BackelinFroeberg1991,MMT92,CKR02}
have found criteria on the leading terms of $G$ that often detect
the property before building the $S$-poly\-nomial, reducing significantly
the number of $S$-poly\-nomials that require inspection.

The authors of \cite{HP05} discovered a new criterion on leading
terms that is useful in some Gr\"obner bases of three polynomials.
In Section~\ref{sec:EFC} we generalize this criterion to Gr\"obner
bases of arbitrary size. The result, called the \emph{Extended Criterion}
(EC), is a new, non-trivial criterion that also extends Buchberger's
criteria. The Main Theorem uses the new criterion to formulate a new
characterization theorem for Gr\"obner bases. In Section~\ref{sec:Proof}
we prove the Main Theorem. In Section~\ref{sec: class of systems}
we identify a class of polynomial systems where Buchberger's Criteria
have no effect, whereas EC reduces the maximum number of $S$-poly\-nomials
required to answer question~(1) from $m\left(m-1\right)/2$ to $m-1$.

\section{\label{sec:EFC}The Extended Criterion}

We begin with a review of the essential notation and background material.
Standard references in the theory of Gr\"obner bases are \cite{BWK93,Adams94,CLO97}.

Fix a commutative ring $\polyring$ of polynomials in $x_{1}$, $x_{2}$,
\ldots{}, $x_{n}$ over a field, and an admissible term ordering~$\ordering$
over the terms of $\polyring$. (In this paper, a term is a monomial
whose coefficient is 1.) For any non-zero $p\in\polyring$, we denote
the leading term of $p$ with respect to $\ordlt$ by $\lt{p}$, and
the leading coefficient by $\lc{p}$.

\begin{dfn}[Gr\"obner Basis]\label{def: GB}We say that $G\in\polyring^{m}$
is a \emph{Gr\"obner basis with respect to $\ordlt$} if for every
polynomial~$p$ in the ideal $I$ generated by $G$ there exists
some $g\in G$ such that $\lt{g}\mid\lt{p}$.\closer
\end{dfn}
Gr\"obner bases provide an elegant framework that allows one to decide
easily many otherwise difficult problems in commutative algebra and
algebraic geometry \cite{Buchberger70,BWK93,CLO97,CLO98,KR00}. From
an algorithmic perspective, however, Definition~\ref{def: GB} is
not useful; after all, $p$ ranges over the infinite set~$I$, so
it is impossible to decide whether $G$ is a Gr\"obner basis by inspecting
every $p\in I$. Bruno Buchberger launched the theory of Gr\"obner
bases by developing a characterization that requires finitely many
inspections.

Before stating Buchberger's characterization, we need a little more
notation. For any $f,g\in\polyring$, write\[
\sigma_{f,g}=\frac{\lcm\left(\lt{f},\lt{g}\right)}{\lt{f}},\]
and define the $S$\emph{-poly\-nomial} of $f$ and $g$ as\[
\SPoly{f,g}=\lc{g}\sigma_{f,g}f-\lc{f}\sigma_{g,f}g.\]

Let $G\in\polyring^{m}$ and $p\in\polyring$, with $p\neq 0$.
We say that $p$ \emph{reduces to zero with respect to $G$}
if $p=0$ or there exist monomials $q_{1}$, $q_{2}$, \ldots{}, $q_{r}$
    and integers $\nu_{1}$, $\nu_{2}$,
        \ldots{},
        $\nu_{r}\in\left\{ 1,2,\ldots,m\right\} $
such that

\begin{itemize}
\item $p=q_{1}g_{\nu_{1}}+q_{2}g_{\nu_{2}}+\cdots+q_{r}g_{\nu_{r}}$; 
\item $\lt{q_{1}}\lt{g_{\nu_{1}}}$ is a term of $p$; and
\item for $i>1$, each $\lt{q_{i}}\lt{g_{\nu_{i}}}$ is a term of $p-q_{1}g_{\nu_{1}}-q_{2}g_{\nu_{2}}-\ldots-q_{i-1}g_{\nu_{i-1}}$.
\end{itemize}
\noindent If $p\neq0$ and no $\lt{g_{j}}$ divides a term of $p$,
then $p$ \emph{does not reduce to zero with respect to} $G$.

The notions of $S$-poly\-nomials and reduction to zero allowed Buchberger
to formulate the following \cite{Buchberger65}.

\begin{thm}
[Buchberger's Characterization]\label{thm: Buchberger's characterization}Let
$G\in\polyring^{m}$. The following are equivalent.
\begin{lyxlist}{(B)}
\item [{(A)}] $G$ is a Gr\"obner basis with respect to $\ordering$.
\item [{(B)}] For every $i,j$ such that $1\leq i<j\leq m$, $\SPoly{g_{i},g_{j}}$
reduces to zero with respect to $G$.\closer
\end{lyxlist}
\end{thm}
Unlike $p$ in Definition~\ref{def: GB}, $i$ and $j$ in (B) range
over finitely many integers. Moreover, deciding whether a polynomial
reduces to zero with respect to $G$ requires a finite number of steps.
This gives Buchberger's Characterization a decided computational advantage
over Definition~\ref{def: GB}.

Nevertheless, it is usually burdensome to check all the $S$-poly\-nomials.
Buchberger developed two criteria
\cite{Buchberger65,KollreiderBuchberger76}
that modify condition (B) of Buchberger's Characterization:

\begin{thm}
\label{thm:Buchberger's Criteria}Let $G\in\polyring^{m}$. The following
are equivalent.
\begin{lyxlist}{(B)}
\item [{(A)}] $G$ is a Gr\"obner basis with respect to $\ordering$.
\item [{(B)}] For every $i,j$ such that $1\leq i<j\leq m$, one of the
following holds:

\begin{lyxlist}{(B3)}
\item [{(B0)}] $\SPoly{g_{i},g_{j}}$ reduces to zero with respect to $G$.
\item [{(B1)}] $\lt{g_{i}}$ and $\lt{g_{j}}$ are relatively prime.
\item [{(B2)}] There exist $k_{1},\ldots,k_{n}$ such that $i=k_{1}$,
$j=k_{n}$,\\
each of the $\lt{g_{k_{\ell}}}$ divides $\lcm\left(\lt{g_{i}},\lt{g_{j}}\right)$,
and\\
each $\SPoly{g_{k_{\ell}},g_{k_{\ell+1}}}$ reduces to zero with respect
to $G$.\closer
\end{lyxlist}
\end{lyxlist}
\end{thm}
These criteria, along with adaptations of them, are widely used in
both decision and transformation
\cite{Buchberger85,GM88,BackelinFroeberg1991,MMT92,CKR02}.
On this account, we make the following definition.

\begin{dfn}
[Buchberger's Criteria]Let $t_{1}$, $t_{2}$, and $t_{3}$ be terms
of $\polyring$. If $t_{1}$ and $t_{2}$ are relatively prime, we
say that $\left(t_{1},t_{2}\right)$ satisfies \emph{Buchberger's
gcd Criterion}. If $t_{2}\mid\lcm\left(t_{1},t_{3}\right)$, we
say that $\left(t_{1},t_{2},t_{3}\right)$ satisfies \emph{Buchberger's
lcm Criterion}.\closer
\end{dfn}

A number of researchers have studied how to apply Buchberger's Criteria
as efficiently as possible \cite{GM88,CKR02}. The algorithm described
by Gebauer and M\"oller is considered a standard benchmark algorithm
for approaches to question (2) posed in the introduction.

The main contribution of this paper is to introduce the following
criterion, which addresses question (1) by means of a new characterization
theorem (the Main Theorem) as well as the identification of a class
of polynomial systems for which the criterion gives a dramatic reduction
in the number of $S$-poly\-nomials required to answer the question
(Section~\ref{sec: class of systems}).

\begin{dfn}
[The Extended Criterion]\label{def: EFC}Let $t_{1}$, \ldots{},
$t_{m}$ be terms of $\polyring$. We say that $\left(t_{1},\ldots,t_{m}\right)$
satisfies \emph{the Extended Criterion} (EC) if it satisfies (EDiv)
and (EVar) where
\begin{lyxlist}{(EDiv)}
\item [{(EDiv)}] for every $k$ such that $1\leq k\leq m$, $\gcd\left(t_{1},t_{m}\right)$
divides $t_{k}$; and
\item [{(EVar)}] for every variable $x$, $\deg_{x}\gcd\left(t_{1},t_{m}\right)=0$
or $\left\{ \deg_{x}t_{k}\right\} _{k=1}^{m}$ is a monotonic sequence.\closer
\end{lyxlist}
\end{dfn}

Observe that $\left(t_1,t_2,\ldots,t_m\right)$ satisfies the Extended Criterion
if and only if its reversal $\left(t_m,t_{m-1},\ldots,t_1\right)$ does.
Hence (EVar) tests for ``monotonic'' without reference to a direction.

\begin{example}
\label{exa: EFC def}The list $T_{1}=\left(x_{0}x_{1},x_{0}x_{2},\ldots,x_{0}x_{m}\right)$
satisfies (EC). Why? (EDiv) is satisfied because $x_{0}$ divides
$t_{k}$ for $k=1,\ldots,m$, and (EVar) is satisfied because $\left\{ \deg_{x_{0}}t_{k}\right\} _{k=1}^{m}=\left(1,1,\ldots,1\right)$
and $\deg_{x_{i}}\gcd\left(t_{1},t_{m}\right)=0$ for $i=1,\ldots,m$.
Observe that no pair or triplet of terms in $T$ satisfies either
of Buchberger's Criteria.

Similarly, the list $T_{2}=\left(x_{0}x_{1},x_{0}^{2}x_{2},x_{0}^{2}x_{3},x_{0}^{3}x_{4}\right)$
satisfies (EC) without satisfying Buchberger's Criteria, as illustrated
by Figure~\ref{fig: Diagram of EFC good}: $\gcd\left(t_{1},t_{4}\right)=x_{0}$
divides both $t_{2}$ and $t_{3}$, and $\left\{ \deg_{x_{0}}t_{k}\right\} _{k=1}^{4}=\left(1,2,2,3\right)$
is monotonic.%
\begin{figure}
\begin{centering}
    \includegraphics[scale=0.4]{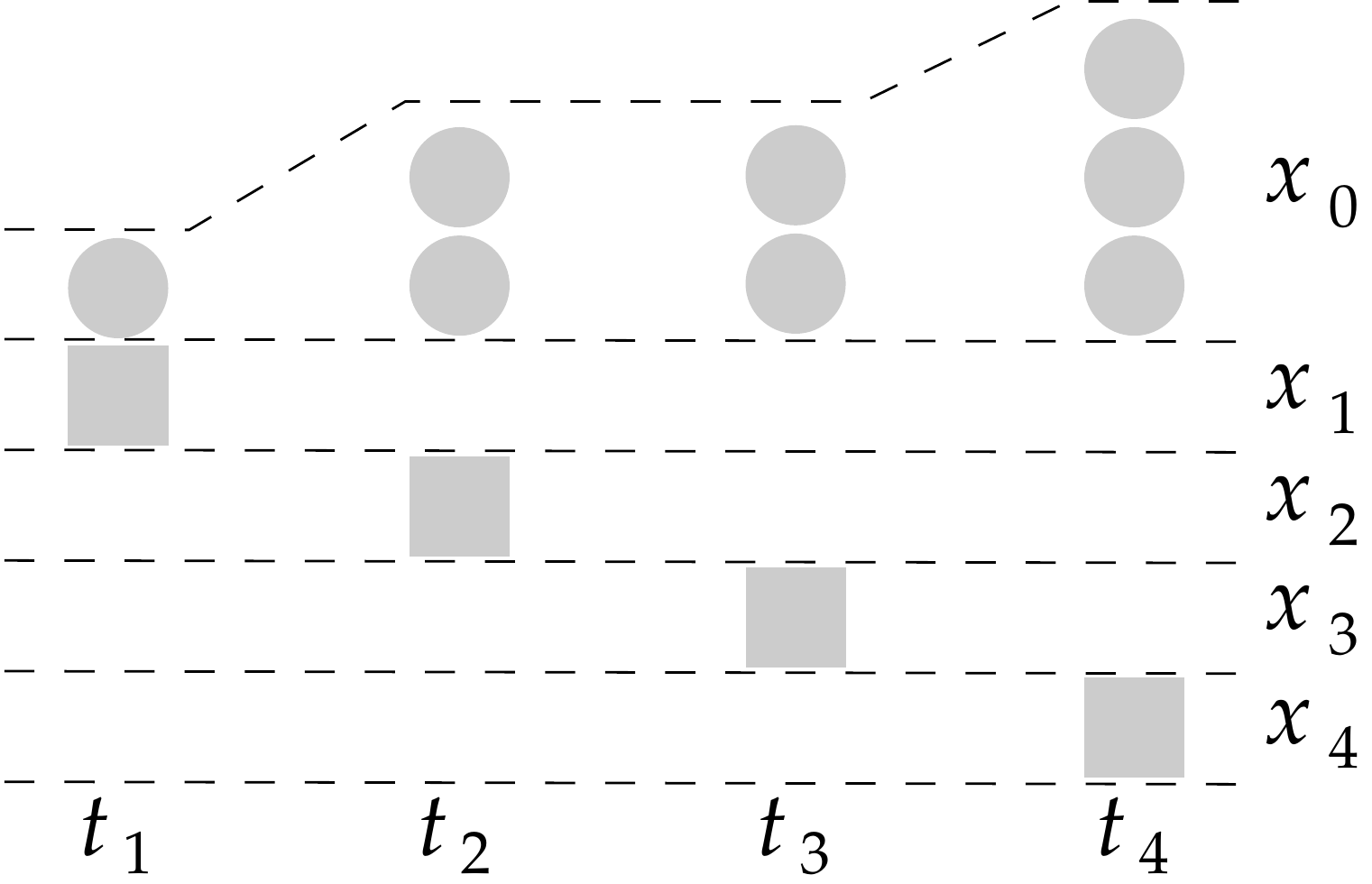}
\par\end{centering}
\caption{\label{fig: Diagram of EFC good}A list of terms that does not satisfy
Buchberger's Criteria, but satisfies the Extended Criterion. Observe
that $\gcd\left(t_{1},t_{4}\right)$ divides $t_{2}$ and $t_{3}$,
and $\left\{ \deg_{x_{0}}t_{i}\right\} _{i=1}^{4}$ is monotonic.
}
\par\rule{\linewidth}{0.1pt}\par
\end{figure}

On the other hand, the list $T_{3}=\left(x_{0}x_{1},x_{0}^{2}x_{2},x_{0}^{3}x_{3},x_{0}^{2}x_{4}\right)$
does not satisfy (EC), because (EVar) is violated: $\left\{ \deg_{x_{0}}t_{k}\right\} _{k=1}^{m}=\left(1,2,3,2\right)$
is not monotonic. This is illustrated by Figure~\ref{fig: Diagram of EFC bad}.%
\begin{figure}
\begin{centering}
\includegraphics[scale=0.4]{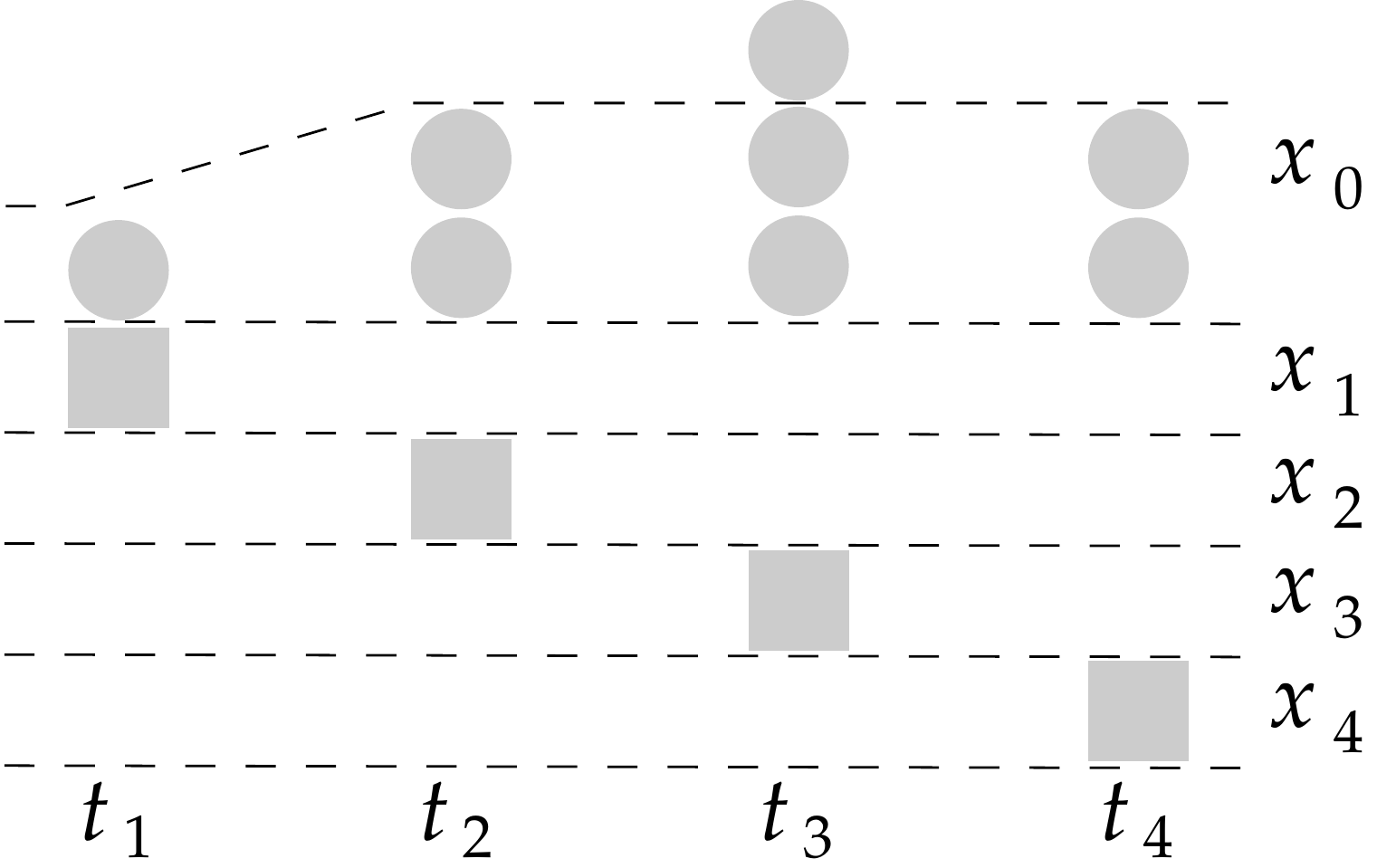}
\par\end{centering}
\caption{\label{fig: Diagram of EFC bad}A list of terms that satisfy neither
Buchberger's Criteria nor the Extended Criterion. Observe that although
$\gcd\left(t_{1},t_{4}\right)$ divides $t_{2}$ and $t_{3}$, $\left\{ \deg_{x_{0}}t_{i}\right\} _{i=1}^{4}$
is not monotonic.}
\par\rule{\linewidth}{0.1pt}\par
\end{figure}
 A permutation of $T_{3}$, $\left(x_{0}x_{1},x_{0}^{2}x_{2},x_{0}^{2}x_{4},x_{0}^{3}x_{3}\right)$,
would satisfy (EC), but such permutations are not always possible
if $t_{1}$ and $t_{m}$ share more than one variable;
consider $\left(x_{1}yz,x_{2}y^{2}z,x_{3}yz^{2},x_{4}y^{3}z^{2},x_{5}yz\right)$.\closer
\end{example}
We can use the Extended Criterion to generalize Buchberger's Characterization
Theorem.

\begin{maintheorem}Let $G\in\polyring^{m}$. The following are equivalent.

\begin{lyxlist}{(B)}
\item [{(A)}] $G$ is a Gr\"obner basis with respect to $\ordering$.
\item [{(B)}] For every $i,j$ such that $1\leq i<j\leq m$, one of the
following holds:

\begin{lyxlist}{(B3)}
\item [{(B0)}] $\SPoly{g_{i},g_{j}}$ reduces to zero with respect to $G$.
\item [{(B1)}] $\lt{g_{i}}$ and $\lt{g_{j}}$ are relatively prime.
\item [{(B2)}] There exist $k_{1},\ldots,k_{n}$ such that $i=k_{1}$,
$j=k_{n}$,\\
each of the $\lt{g_{k_{\ell}}}$ divides $\lcm\left(\lt{g_{i}},\lt{g_{j}}\right)$,
and\\
each $\SPoly{g_{k_{\ell}},g_{k_{\ell+1}}}$ reduces to zero with respect
to $G$.
\item [{(B3)}] There exist $k_{1},\ldots,k_{n}$ such that $i=k_{1}$,
$j=k_{n}$,\\
the list of leading terms of $g_{k_{1}},\ldots,g_{k_{n}}$ satisfy
EC, and\\
each $\SPoly{g_{k_{\ell}},g_{k_{\ell+1}}}$ reduces to zero with respect
to $G'=\left(g_{k_{1}},\ldots,g_{k_{n}}\right)$.\closer
\end{lyxlist}
\end{lyxlist}
\end{maintheorem}

It is essential that in (B3), the reductions to zero are with respect
to $G'$ and not to $G$. If we use $G$ instead of $G'$, then we
may not have a Gr\"obner basis; see Example~\ref{exa: bad application}.
This also makes it a bad idea to try to combine (B3) and (B2) into
one disjunction.

If the terms $t_{1}$ and $t_{m}$ are relatively prime, then $\left(t_{1},\ldots,t_{m}\right)$
satisfies (EDiv) and (EVar) easily. Hence, pairs of leading terms
that satisfy Buchberger's gcd Criterion also satisfy the Extended
Criterion. However, it is not easy to condense (B1) and (B3) into
one criterion, because (B3) requires that a chain of $S$-poly\-nomials
reduce to zero, while (B1) does not.

When $m=3$, EC is equivalent to the criterion of \cite{HP05}, which
generalizes \emph{both} of Buchberger's Criteria. For $m>3$, this
is not the case! Terms can satisfy Buchberger's lcm Criterion without
satisfying EC, and as in Example~\ref{exa: EFC def}, terms can satisfy
EC without satisfying Buchberger's lcm Criterion.

The remainder of this section consists of examples:

\begin{itemize}
\item Example~\ref{exa: easy example} provides a straightforward application
of the Main Theorem;
\item Example~\ref{exa: bad application} shows an invalid application
of the Main Theorem.
\end{itemize}
\begin{example}
\label{exa: easy example}Let $G=\left(g_{1},g_{2},g_{3},g_{4}\right)$
where\begin{align*}
g_{1} & =4x_{0}x_{1}+2x_{0}x_{2}+3x_{0}x_{4}-8x_{1}-4x_{2}-6x_{4}\\
g_{2} & =3x_{0}^{2}x_{2}+2x_{0}^{2}x_{4}-6x_{0}x_{2}-4x_{0}x_{4}\\
g_{3} & =4x_{0}^{2}x_{3}+2x_{0}^{2}x_{4}-8x_{0}x_{3}-4x_{0}x_{4}\\
g_{4} & =2x_{0}^{3}x_{4}-2x_{0}^{2}x_{3}-x_{0}^{2}x_{4}+4x_{0}x_{3}-6x_{0}x_{4}.\end{align*}
Let $\ordering$ represent any term ordering such that
$\lt{g_1}=x_0x_1$, $\lt{g_2}=x_0^2x_2$, $\lt{g_3}=x_0^2x_3$,
and $\lt{g_4}=x_0^3x_4$.
We pose this question: \emph{Is $G$ a Gr\"obner basis with respect to $\ordlt$?}

Routine computation verifies that the pairs $\left(1,2\right)$, $\left(2,3\right)$,
and $\left(3,4\right)$ satisfy (B0) of Theorem~\ref{thm:Buchberger's Criteria}
and of the Main Theorem;
that is, $\SPoly{g_{1},g_{2}}$, $\SPoly{g_{2},g_{3}}$, and $\SPoly{g_{3},g_{4}}$
reduce to zero with respect to $G$.
We can say something more: in the process of reducing them,
we discover that for $i=1,2,3$ each $\SPoly{g_i,g_{i+1}}$ reduces to zero
with respect to $\{g_i,g_{i+1}\}$. This will prove important in a moment.

As for the remaining pairs, they do not satisfy (B1) or (B2) of either theorem,
because no permutation of the leading terms $x_{0}x_{1}$, $x_{0}^{2}x_{2}$,
$x_{0}^{2}x_{3}$, and $x_{0}^{3}x_{4}$ satisfies Buchberger's criteria.
Thus, Theorem~\ref{thm:Buchberger's Criteria} does not help us answer
the question posed.

However, the Main Theorem does.
Observe that \[\left(\lt{g_{1}},\lt{g_{2}},\lt{g_{3}},\lt{g_{4}}\right)=T_{2}\]
where $T_2$ was defined in Example~\ref{exa: easy example}; the Extended Criterion applies to $T_2$.
In addition, $\SPoly{g_{1},g_{2}}$, $\SPoly{g_{2},g_{3}}$, and
$\SPoly{g_{3},g_{4}}$ reduce to zero with respect to $G$.
Hence $(1,4)$ satisfies (B3) of the Main Theorem with $G'=G$.

We are not quite done: to decide whether $G$ is a Gr\"obner basis,
we must resolve the pairs $(1,3)$ and $(2,4)$.
The Main Theorem shows that these pairs also satisfy (B0).
\begin{itemize}
\item To show that $\SPoly{g_1,g_3}$ reduces to zero, we claim that $\{g_1,g_2,g_3\}$ is a Gr\"obner basis:
\begin{itemize}
\item We know that the pairs $(1,2)$ and $(2,3)$ satisfy (B0) of the Main Theorem.
\item The Extended Criterion applies to $\left(\lt{g_{1}},\lt{g_{2}},\lt{g_{3}}\right)$.
\item Recalling that each $\SPoly{g_i,g_{i+1}}$ reduces to zero w.r.t.~$\{g_i,g_{i+1}\}$,
we infer that $\SPoly{g_1,g_2}$ and $\SPoly{g_2,g_3}$ reduce to zero w.r.t.~$G^{(1,2,3)}=\left(g_1,g_2,g_3\right)$.
Thus the pair $(1,3)$ satisfies (B3) of the Main Theorem.
\item This implies that $G^{(1,2,3)}$ is a Gr\"obner basis, so $\SPoly{g_1,g_3}$ reduces to zero.
\end{itemize}
\item To show that $\SPoly{g_2,g_4}$ reduces to zero, we claim that $\{g_2,g_3,g_4\}$ is a Gr\"obner basis:
\begin{itemize}
\item We know that the pairs $(2,3)$ and $(3,4)$ satisfy (B0) of the Main Theorem.
\item The Extended Criterion applies to $\left(\lt{g_{2}},\lt{g_{3}},\lt{g_{4}}\right)$.
\item Recalling that each $\SPoly{g_i,g_{i+1}}$ reduces to zero w.r.t.~$\{g_i,g_{i+1}\}$,
we infer that $\SPoly{g_2,g_3}$ and $\SPoly{g_3,g_4}$ reduce to zero w.r.t.~$G^{(2,3,4)}=\left(g_2,g_3,g_4\right)$.
Thus the pair $(2,4)$ satisfies (B3) of the Main Theorem.
\item This implies that $G^{(2,3,4)}$ is a Gr\"obner basis, so $\SPoly{g_2,g_4}$ reduces to zero.
\end{itemize}
\end{itemize}
\noindent Recall that $(1,4)$ satisfies (B3) of the Main Theorem with $G'=G$.
We now know that the other pairs satisfy (B0).
It follows from the Main Theorem that $G$ is indeed a Gr\"obner basis with respect
to $\ordlt$. We have answered the question posed by reducing
only three of the six $S$-poly\-nomials to zero.

To achieve this, we had to know not only that the $S$-poly\-nomials
reduced to zero, but also over which subsets of $G$ they were reduced!
Had those subsets been different, the Extended Criterion probably would not apply,
as Example~\ref{exa: bad application} shows below.
Conversely, it is conceivable that one could apply the Extended Criterion
but not realize it, because one has verified that the $S$-poly\-nomials in question
reduce to zero with respect to a different subset of $G$ than the one needed.\closer
\end{example}
The following example illustrates why (B3) of the Main Theorem requires
$G'$ and not $G$.

\begin{example}
\label{exa: bad application}Let $G=\left(g_{1},g_{2},g_{3},g_{4}\right)$
where\begin{align*}
g_{1} & =x^{2}y+z\\
g_{2} & =xyz\\
g_{3} & =xy^{2}\\
g_{4} & =z^{2}.\end{align*}
Let $\ordering$ be any ordering such that $x^{2}y\ordgt z$. Again
we ask, \emph{Is $G$ a Gr\"obner basis with respect to $\ordlt$?}

It is easy to verify that pairs $\left(1,2\right)$, $\left(1,4\right)$,
$\left(2,3\right)$, $\left(2,4\right)$, and $\left(3,4\right)$
satisfy (B0) of the Main Theorem. The leading terms of $g_{1}$, $g_{2}$,
and $g_{3}$ satisfy the Extended Criterion, so set $G'=\left(g_{1},g_{2},g_{3}\right)$.
A subquestion: Does (B3) of the Main Theorem imply that $G$ is a
Gr\"obner basis? No, because the $S$-poly\-nomials $\SPoly{g_{1},g_{2}}$
and $\SPoly{g_{2},g_{3}}$ reduce to zero with respect to $G$, but
not with respect to $G'$. In fact, $\SPoly{g_{1},g_{3}}=yz$ \emph{does
not} reduce to zero with respect to $G$ even though all the other
$S$-poly\-nomials do! Thus $G$ is not a Gr\"obner basis with respect
to $\ordlt$.\closer
\end{example}

\section{\label{sec:Proof}Proof of the Main Theorem}

Before diving into details, we pause a moment to describe the fundamental
goal of the proof. A previous example will serve us well. The polynomials
of Example~\ref{exa: easy example} factor as follows:\begin{align*}
g_{1} & =\phantom{2x_{0}}\left(x_{0}-2\right)\left(4x_{1}+2x_{2}+3x_{4}\right)\\
g_{2} & =\phantom{2}x_{0}\left(x_{0}-2\right)\left(3x_{2}+2x_{4}\right)\\
g_{3} & =2x_{0}\left(x_{0}-2\right)\left(2x_{3}+x_{4}\right)\\
g_{4} & =\phantom{2}x_{0}\left(x_{0}-2\right)\left(2x_{0}x_{4}+3x_{4}-2x_{3}\right).\end{align*}
Any pair of the polynomials has a common divisor whose cofactors have
relatively prime leading terms: for example, the common divisor of
$g_{1}$ and $g_{4}$ is $x_{0}-2$, and the leading terms of the
cofactors are $x_{1}$ and $x_{0}^{2}x_{4}$, respectively. From (B1)
of Theorem~\ref{thm:Buchberger's Criteria}, we know that the system
of cofactors of the gcd is a Gr\"obner basis. Generating a new system
whose polynomials are multiples of the cofactors does not alter this,
\emph{provided that} for each pair the multiple of the cofactors is
common.

The fundamental goal of the proof is to generalize this observation.
Theorem~\ref{thm: completion of sufficiency} accomplishes this.
Lemma~\ref{lem: Chain implies certain lts} is a technical lemma
that fills in a crucial step of Lemma~\ref{lem: Chain implies gcd commutes},
which in its turn is a technical lemma that fills in a crucial step
of Theorem~\ref{thm: completion of sufficiency}.
Lemmas~\ref{lem:remove some variables from system} and~\ref{lem:matrix nonsingular}
are also technical lemmas that help clarify
some linear algebra  necessary for the proof of Lemma~\ref{lem: Chain implies certain lts}.

Although Lemmas~\ref{lem: Chain implies gcd commutes} and~\ref{thm: completion of sufficiency}
generalize similar lemmas in \cite{HP05}, the increased size of the
list ($m>3$) required the development of the entirely new Lemma~\ref{lem: Chain implies certain lts},
as well as substantial changes to the proof of Lemma~\ref{lem: Chain implies gcd commutes}.
In addition, Theorem~\ref{thm: completion of sufficiency} leads
to the important consequence Corollary~\ref{cor:EFC implies reprime lts};
this consequence went unremarked in the previous work, but will show
itself useful in Section~\ref{sec: class of systems}.

Besides a proof of the main theorem, this section develops several
results that are interesting or useful in other contexts. Lemma~\ref{lem: Chain implies certain lts},
for example, took us completely by surprise.
Lemma \ref{lem: Chain implies gcd commutes} generalizes a relationship between
the gcd of two polynomials and their $S$-poly\-nomial. 
Theorem~\ref{thm: completion of sufficiency}
is similar to a well-known theorem regarding Buchberger's lcm Criterion;
it will prove useful in Section~\ref{sec: class of systems}, whereas
the Main Theorem does not.

We turn to the proof. We regularly make implicit use of Proposition~\ref{pro: properties of lts}
below. The proof is easy and well-known, so we do not repeat it here.

\begin{prop}
\label{pro: properties of lts}For all $f,g\in\polyring$ each of the following holds.

(A) If $f+g\neq0$, then $\lt{f+g}\ordle\max_{\ordlt}\left(\lt{f},\lt{g}\right)$.

(B) $\lt{f\cdot g}=\lt{f}\cdot\lt{g}$.

(C) If $f/g$ is a polynomial, then $\lt{f/g}=\lt{f}/\lt{g}$.\closer
\end{prop}
At this point we introduce the concept of an $S$-re\-pre\-sen\-ta\-tion,
which is essential to the proof.

\begin{dfn}
\label{def: valid SPoly rep}Let $p\in\polyring$, $t$ a term of
$\polyring$, and $G\in\polyring^{m}$. We say that $\mathbf{h}\in\polyring^{m}$
\emph{is a }$t$\emph{-re\-pre\-sen\-ta\-tion of}\textbf{\emph{
$p$ }}\emph{with respect to} $G$ if $p=h_{1}g_{1}+\cdots+h_{m}g_{m}$
and for all $i$ such that $1\leq i\leq m$, we have $h_{i}=0$ or
$\lt{h_{i}g_{i}}\ordle t$.

Furthermore, let $g_{i},g_{j}\in G$. If $t\ordlt\lcm\left(\lt{g_{i}},\lt{g_{j}}\right)$
and $\mathbf{h}$ is a $t$-re\-pre\-sen\-ta\-tion of $\SPoly{g_{i},g_{j}}$
with respect to $G$, then we say that $\SPoly{g_{i},g_{j}}$ \emph{has}\textbf{\emph{
}}\emph{an $S$-re\-pre\-sen\-ta\-tion with respect to} $G$,
and that $\mathbf{h}$ \emph{is an $S$-re\-pre\-sen\-ta\-tion
of} $\SPoly{g_{i},g_{j}}$\emph{ with respect to} $G$. We may omit
``with respect to $G$'' if it is clear from the context.\closer
\end{dfn}
The notion of $S$-representation is related, but not equivalent,
to the notion of reduction to zero. We discuss this relationship
near the end of the section, where it becomes important for the Main
Theorem. For the time being, we content ourselves with exploring how
the Extended Criterion can link a chain of $S$-re\-pre\-sen\-ta\-tions.

To do that, we will need Lemma~\ref{lem: Chain implies certain lts},
which identifies a useful and interesting structure in a certain chain
of $S$-re\-pre\-sen\-ta\-tions.

\begin{lem}
\label{lem: Chain implies certain lts}Let $G\in\polyring^{m}$. Then
(A) $\Longrightarrow$ (B) where
\begin{lyxlist}{(B)}
\item [{(A)}] $\SPoly{g_{1},g_{2}}$, $\SPoly{g_{2},g_{3}}$, \ldots{},
and $\SPoly{g_{m-1},g_{m}}$ all have $S$-re\-pre\-sen\-ta\-tions
with respect to $G$.
\item [{(B)}] There exist $P,Q\in\polyring$ such that $P\cdot g_{1}=Q\cdot g_{m}$ and

\begin{lyxlist}{(B)}
\item [{~}] $\lt{P}=\sigma_{g_{1},g_{2}}\sigma_{g_{2},g_{3}}\cdots\sigma_{g_{m-1},g_{m}}$,
and
\item [{~}] $\lt{Q}=\sigma_{g_{2},g_{1}}\sigma_{g_{3},g_{2}}\cdots\sigma_{g_{m},g_{m-1}}$.\closer
\end{lyxlist}
\end{lyxlist}
\end{lem}
The proof of Lemma \ref{lem: Chain implies certain lts} requires some non-trivial linear algebra,
so we defer it to page~\pageref{pf:Proof of lem matrix nonsingular}.
Lemmas~\ref{lem:remove some variables from system} and~\ref{lem:matrix nonsingular}
provide the necessary results.
Lemma~\ref{lem:remove some variables from system} describes a relationship between
the elimination of variables in a linear system and the coefficients of those variables.

\begin{lem}\label{lem:remove some variables from system}
Let $n\in\mathbb{N^+}$. Consider the system of $n-1$ linear equations in $n$ variables\[
    \mathcal{S}_1 = \left\{\sum_{j=1}^na_{i,j}x_j\right\}_{i=1}^{n-1}.
\]
For $k=1,\ldots,n-2$ define the matrix\[
    A_k = \left(\begin{array}{cccc}
            a_{1,1} & a_{1,2} & \cdots & a_{1,k} \\
            a_{2,1} & a_{2,2} & \cdots & a_{2,k} \\
            \vdots & & \ddots & \vdots \\
            a_{k,1} & a_{k,2} & \cdots & a_{k,k}
        \end{array}\right).
\]
If each $A_k$ has nonzero determinant,
then for each $k=2,\ldots,n-1$ the system\[
    \mathcal{S}_k = \left\{\sum_{j=i}^n b_{i,j}^{(k)} x_j =0\right\}_{i=k}^{n-1}
\]
with\[
    b_{i,j}^{(k)} = \left|\begin{array}{cc}
            A_{k-1} & \begin{array}{c}
                    a_{1,j} \\ a_{2,j} \\ \vdots \\ a_{k-1,j}
                \end{array} \\
            \begin{array}{cccc}
                a_{i,1} & a_{i,2} & \cdots & a_{i,k-1}
            \end{array} & a_{i,j}
        \end{array}\right|
\] is consistent.\closer
\end{lem}

To prove Lemma~\ref{lem:remove some variables from system},
we use the following special case of Jacobi's Theorem on determinants,
whose proof we do not reproduce here \cite{Jacobi1833,DodgsonsMethod}.

\begin{thm}\label{thm: Jacobi's Theorem}
Let $A$ be an $n\times n$ matrix, $M$ a $2\times 2$ minor of $A$,
$M'$ the corresponding $2\times 2$ minor of the adjugate of $A$,
and $M^*$ the $\left(n-2\right)\times\left(n-2\right)$ minor of $A$
that is complementary to $M$.
Then\[\det M' = \det A\cdot\det M^*.\closer\]
\end{thm}

We will use Theorem~\ref{thm: Jacobi's Theorem} by putting $M$ as the corners of the matrix,
making $M^*$ the interior.

\begin{proof}[Proof of Lemma \ref{lem:remove some variables from system}]
We proceed by induction on $k$.
For the inductive base $k=2$,
eliminate $x_1$ from equations $i=2,\ldots,n-1$ in $\mathcal{S}_1$
by subtracting the product of the first equation and $a_{i,1}$
from the product of the second equation and $a_{1,1}$.
It is routine to verify that for $i=2,\ldots,n-1$ and $j=2,\ldots,n$ we have \[
    b_{i,j}^{(k)} = \left|\begin{array}{cc}
            a_{1,1} & a_{1,j} \\
            a_{i,1} & a_{i,j}
        \end{array}\right|.
\]

Now assume the assertion is true for all $\ell$ where $1\leq\ell<k$.
In system $\mathcal{S}_{k-1}$ use equation $k-1$ to eliminate the variable $x_{k-1}$
from equations $k$, \ldots, $n-1$.
We obtain a new system of equations\[
    \mathcal{S}_k = \left\{\sum_{j=i}^n \beta_{i,j} x_j =0\right\}_{i=k}^{n-1}
\]
where for each $i,j,k$ we have
\begin{align*}
    \beta_{i,j} &= \left|\begin{array}{cc}
                b_{k-1,k-1}^{(k-1)} & b_{k-1,j}^{(k-1)} \\
                \\
                b_{i,k-1}^{(k-1)} & b_{i,j}^{(k-1)}
            \end{array}\right| \\
        &= \left|\begin{array}{cc}
            A_{k-2} & \begin{array}{c}
                    a_{1,k-1} \\ \vdots \\ a_{k-2,k-1}
                \end{array} \\
                a_{k-1,1} \quad \cdots \quad a_{k-1,k-2}
            & a_{k-1,k-1}
        \end{array}\right| \left|\begin{array}{cc}
            A_{k-2} & \begin{array}{c}
                    a_{1,j} \\ \vdots \\ a_{k-2,j}
                \end{array} \\
                a_{i,1} \quad \cdots \quad a_{i,k-2}
            & a_{i,j}
        \end{array}\right|\\&\quad - \left|\begin{array}{cc}
            A_{k-2} & \begin{array}{c}
                    a_{1,k-1} \\ \vdots \\ a_{k-2,k-1}
                \end{array} \\
                a_{i,1} \quad \cdots \quad a_{i,k-2}
            & a_{i,k-1}
        \end{array}\right| \left|\begin{array}{cc}
            A_{k-2} & \begin{array}{c}
                    a_{1,j} \\ \vdots \\ a_{k-2,j}
                \end{array} \\
                a_{k-1,1} \quad \cdots \quad a_{k-1,k-2}
            & a_{k-1,j}
        \end{array}\right|.
\end{align*}
Perform the following row and column swaps:\begin{itemize}
    \item in $b_{k-1,k-1}^{(k-1)}$, move the bottom row to the top, and the rightmost row to the leftmost;
    \item in $b_{k-1,j}^{(k-1)}$, do nothing;
    \item in $b_{i,k-1}^{(k-1)}$, move the rightmost row to the leftmost; and
    \item in $b_{i,j}^{(k-1)}$, move the bottom row to the top.
\end{itemize}
Denote the resulting matrices by $B_1$, $B_2$, $B_3$, and $B_4$;
the negatives introduced by the row and column swap cancel, so that $\beta_{i,j}=B_1B_2 - B_3B_4$.

Let \[
    C = \left|\begin{array}{ccccc}
            a_{k-1,k-1} & a_{k-1,1} & \cdots & a_{k-1,k-2} & a_{k-1,j} \\
            a_{1,k-1} & & & & a_{1,j} \\
            \vdots & & A_{k-2} & & \vdots \\
            a_{k-2,k-1} & & & & a_{k-2,j} \\
            a_{i,k-1} & a_{i,1} & \cdots & a_{i,k-2} & a_{i,j}
        \end{array}\right|.
\]
Theorem~\ref{thm: Jacobi's Theorem} with\[
    M=\left(\begin{array}{cc}a_{k-1,k-1} & a_{k-1,j} \\ a_{i,k-1} & a_{i,j} \end{array}\right)
    \mbox{ and }M^*=A_{k-2}
\]implies that\[
    \beta_{i,j} = \left|C\right| \cdot \left|A_{k-2}\right|.
\]
Move the top row of $C$ to the next-to-last row, and the leftmost row of $C$ to the next-to-last column;
the negatives introduced by the row and column swaps cancel, so that\[
    \beta_{i,j} = \left|\begin{array}{cccc}
              & & & a_{1,j} \\
              & A_{k-1} & & \vdots \\
              & & & a_{k-1,j} \\
             a_{i,1} & \cdots & a_{i,k-1} & a_{i,j}
        \end{array}\right|
        \left|A_{k-2}\right|.
\]From the assumption that $A_{k-2}$ is nonzero,
we can divide each equation of $\mathcal{S}_k$ by $A_{k-2}$,
obtaining the desired linear system.
\end{proof}
From this point on, the presence of several $S$-re\-pre\-sen\-ta\-tions requires
a notation that will allow us to distinguish them.

\begin{notation*}
Let $G\in\polyring^{m}$. Let $i,j\in\left\{ 1,\ldots,m-1\right\} $
be distinct. We write\[
\mathbf{h}^{\left(i,j\right)}=\left(h_{1}^{\left(i,j\right)},h_{2}^{\left(i,j\right)},\ldots,h_{m}^{\left(i,j\right)}\right)\]
for an $S$-re\-pre\-sen\-ta\-tion of $\SPoly{g_{i},g_{j}}$ with
respect to $G$. In addition, when $i<j$ we write\begin{align*}
Z_{i,j} & =-\lc{g_j}\sigma_{g_{i},g_{j}}+h_{i}^{\left(i,j\right)}\\
Z_{j,i} & =\lc{g_i}\sigma_{g_{j},g_{i}}+h_{j}^{\left(i,j\right)}.\end{align*}
Note that $\lt{Z_{i,j}}=\sigma_{g_{i},g_{j}}$ and $\lt{Z_{j,i}}=\sigma_{g_{j},g_{i}}$.\closer
\end{notation*}

In the proof of Lemma~\ref{lem: Chain implies certain lts} we will simplify
a linear system of the form shown in Lemma~\ref{lem:remove some variables from system}.
To perform this simplification,
we must ascertain that the matrices $A_k$ in that context have nonzero determinant.

\begin{lem}\label{lem:matrix nonsingular}
Let $G\in\polyring^m$. Then (A) $\Longrightarrow$ (B) where
\begin{lyxlist}{(B)}
\item[(A)] $\SPoly{g_1,g_2}$, $\SPoly{g_2,g_3}$, \ldots, and $\SPoly{g_{m-1},g_m}$ all have $S$-representations
  with respect to $G$.
\item[(B)] For each $k=2,\ldots,m-1$ the $k\times k$ matrix\[
    A_k = \left(\begin{array}{ccccc}
            Z_{2,1} & h_3^{(1,2)} & h_4^{(1,2)} & \cdots & h_{k+1}^{(1,2)} \\
            Z_{2,3} & Z_{3,2} & h_4^{(2,3)} & \cdots & h_{k+1}^{(2,3)} \\
            h_2^{(3,4)} & \ddots & \ddots & & h_{k+1}^{(3,4)}\\
            \vdots & & \ddots & \ddots & \vdots \\
            h_2^{(k,k+1)} & \cdots & h_{k-2}^{(k,k+1)} & Z_{k,k+1} & Z_{k+1,k}
        \end{array}\right)
\] has nonzero determinant; indeed $\lt{\det A_k} = \sigma_{2,1}\sigma_{3,2}\cdots\sigma_{k+1,k}$.\closer
\end{lyxlist}
\end{lem}
The proof of Lemma~\ref{lem:matrix nonsingular} is tricky,
so we present a simple but nontrivial example to illustrate the strategy.
\begin{example}
Suppose $m>3$ and the system $G\in\polyring^m$ satisfies (A) of Lemma~\ref{lem:matrix nonsingular}.
We show that (B) is satisfied for $k=3$. A determinant is a sum of elementary products; since\[
    A_3 = \left(\begin{array}{ccc}
            Z_{2,1} & h_3^{(1,2)} & h_4^{(1,2)} \\
            Z_{2,3} & Z_{3,2} & h_4^{(2,3)} \\
            h_2^{(3,4)} & Z_{3,4} & Z_{4,3}
        \end{array}\right)
\]and the leading term of $Z_{2,1}Z_{3,2}Z_{4,3}$ is $\tau=\sigma_{2,1}\sigma_{3,2}\sigma_{4,3}$,
the leading term of at least one elementary product of $\det A_3$ has the desired form.

We claim that the leading term of every other elementary product of $\det A_3$ is smaller than $\tau$.
We proceed by way of contradiction.
Assume that some other term in the elementary product has a leading term greater than or equal to $\tau$.
Consider the leading terms of the other five polynomials, denoting $\lcm\left(\lt{g_i},\lt{g_j}\right)$
by $L_{i,j}$ and $\lt{g_i}$ by $t_i$.
\begin{lyxlist}{Case 9:}
\item[Case 1:] Suppose that $\tau\ordle\lt{h_3^{(1,2)}\cdot h_4^{(2,3)}\cdot h_2^{(3,4)}}$.
    Multiply both sides of the inequality by $t_2t_3t_4$ to obtain\[
        L_{1,2}L_{2,3}L_{3,4}
            \ordle \left[t_3\cdot\lt{h_3^{(1,2)}}\right]\left[t_4\cdot\lt{h_4^{(2,3)}}\right]\left[t_2\cdot\lt{h_2^{(3,4)}}\right],
    \]which contradicts the definition of an $S$-representation.
\item[Case 2:] Suppose that $\tau\ordle\lt{h_4^{(1,2)}\cdot Z_{2,3}\cdot Z_{3,4}}$.
    Multiply both sides of the inequality by $t_2t_3t_4$ to obtain\[
        L_{1,2}L_{2,3}L_{3,4}
            \ordle \left[t_4\cdot\lt{h_4^{(1,2)}}\right]\cdot L_{2,3}\cdot L_{3,4},
    \]and divide both sides by the common lcm's to obtain\[
        L_{1,2}\ordle t_4\cdot\lt{h_4^{(1,2)}},
    \]which contradicts the definition of an $S$-representation.
\item[Case 3:] Suppose that $\tau\ordle\lt{h_2^{(3,4)}\cdot Z_{3,2}\cdot h_4^{(1,2)}}$.
    Multiply both sides of the inequality by $t_2t_3t_4$ to obtain\[
        L_{1,2}L_{2,3}L_{3,4}
            \ordle \left[t_2\cdot\lt{h_2^{(3,4)}}\right]\cdot L_{2,3}\cdot\left[t_4\cdot\lt{h_4^{(1,2)}}\right],
    \]and divide both sides by the common lcm to obtain\[
        L_{1,2}L_{3,4}\ordle \left[t_2\cdot\lt{h_2^{(3,4)}}\right]\left[t_4\cdot\lt{h_4^{(1,2)}}\right],
    \]which contradicts the definition of an $S$-representation.
\item[Case 4:] Suppose that $\tau\ordle\lt{Z_{(3,4)}\cdot h_4^{(2,3)}\cdot Z_{2,1}}$.
    Multiply both sides of the inequality by $t_2t_3t_4$ to obtain\[
        L_{1,2}L_{2,3}L_{3,4}
            \ordle L_{3,4}\cdot\left[t_4\cdot\lt{h_4^{(2,3)}}\right]\cdot L_{1,2},
    \]and divide both sides by the common lcm's to obtain\[
        L_{2,3}\ordle t_4\cdot\lt{h_4^{(2,3)}},
    \]which contradicts the definition of an $S$-representation.
\item[Case 5:] Suppose that $\tau\ordle\lt{Z_{4,3}\cdot Z_{2,3}\cdot h_3^{(1,2)}}$.
    Multiply both sides of the inequality by $t_2t_3t_4$ to obtain\[
        L_{1,2}L_{2,3}L_{3,4}
            \ordle L_{3,4}\cdot L_{2,3}\cdot\left[t_3\cdot\lt{h_3^{(1,2)}}\right],
    \]and divide both sides by the common lcm's to obtain\[
        L_{1,2}\ordle t_3\cdot\lt{h_3^{(1,2)}},
    \]which contradicts the definition of an $S$-representation.\closer
\end{lyxlist}
\end{example}
The proof of Lemma \ref{lem:matrix nonsingular} follows this strategy.
It is clear from the main diagonal of each $A_k$
that the leading term $t$ of one elementary product of the determinant of $A_k$ has the desired form;
assume by way of contradiction that the leading term of another elementary product is greater than or equal to $t$;
simplify the equivalent inequality by clearing the denominators and dividing the lcm's;
the resulting inequality will contradict the definition of an $S$-repre\-sen\-tation.

\begin{proof}[Proof of Lemma \ref{lem:matrix nonsingular}]\label{pf:Proof of lem matrix nonsingular}
It is clear that $\det A_k$ is a polynomial, each of whose terms is an elementary product of the matrix.
We can write any elementary product as $T=\prod_{i=1}^k B_i$
such that\begin{itemize}
\item each $B_i$ is an element of row $i$; and
\item if $i\neq j$ then $B_i$ and $B_j$ are elements of different columns.
\end{itemize}
As noted above, the main diagonal $A_k$ produces an elementary product whose leading term has the desired form;
we claim that every other elementary product has a smaller leading term.

We proceed by way of contradiction. Assume that some elementary product $T$ besides the main diagonal
satisfies \begin{equation}\label{eq:elementary product}\prod_{i=1}^k\sigma_{i+1,i}\ordle\lt T.\end{equation}
Partition the set of factors of $T$ into three sets:\begin{itemize}
\item $\mathcal D$, containing those factors which are on the main diagonal,
    which have the form $Z_{i+1,i}$ for some $i=1,\ldots,k$;
\item $\mathcal L$, containing those factors which are immediately below the main diagonal,
    which have the form $Z_{i,i+1}$ for some $i=2,\ldots,k$; and
\item $\mathcal O$, containing the other factors, which have the form $h_i^{(j,j+1)}$ for appropriate $i,j$.
\end{itemize}
Since $T$ is not the product of the main diagonal,
the uniqueness of row and column representatives among the factors of $T$
implies that $\mathcal O$ is guaranteed to be nonempty.

Denote $\lcm\left(\lt{g_i},\lt{g_j}\right)$ by $L_{i,j}$ and $\lt{g_i}$ by $t_i$.
Multiply both sides of \eqref{eq:elementary product} by $\prod_{\ell=2}^{k+1}t_\ell$.
This results in the equation\begin{equation*}
\prod_{i=1}^kt_{i+1}\cdot \sigma_{i+1,i}\ordle
    \prod_{\ell=2}^{k+1}t_\ell \cdot
    \prod_{Z_{i+1,i}\in\mathcal D}\sigma_{i+1,i} \cdot
    \prod_{Z_{i,i+1}\in\mathcal L}\sigma_{i,i+1} \cdot
    \prod_{h_i^{(j,j+1)}\in\mathcal O}h_i^{(j,j+1)}.
\end{equation*}
Simplify the left hand side to obtain
\begin{equation}\label{eq:simplified left}
\prod_{i=1}^kL_{i,i+1}\ordle
    \prod_{\ell=2}^{k+1}t_\ell \cdot
    \prod_{Z_{i+1,i}\in\mathcal D}\sigma_{i+1,i} \cdot
    \prod_{Z_{i,i+1}\in\mathcal L}\sigma_{i,i+1} \cdot
    \prod_{h_i^{(j,j+1)}\in\mathcal O}h_i^{(j,j+1)}.
\end{equation}
Rearrange the right hand side of \eqref{eq:simplified left}
by pairing each $t_\ell$ with the corresponding factor taken from column $\ell-1$.
The uniqueness of column representatives among the factors of an elementary product of a matrix guarantees a one-to-one pairing.
If $t_\ell$ is paired with an element of\begin{itemize}
\item $\mathcal D$, it is paired with $Z_{\ell,\ell-1}$, and the product simplifies to $L_{\ell-1,\ell}$;
\item $\mathcal L$, it is paired with $Z_{\ell,\ell+1}$, and the product simplifies to $L_{\ell,\ell+1}$;
\item if $t_\ell$ is paired with an element of $\mathcal O$, it is paired with $h_\ell^{(j,j+1)}$
for appropriate $j$.
\end{itemize}
In addition, the uniqueness of row representatives among the factors of an elementary product
implies that for each $i$, at most one pairing simplifies to $L_{i,i+1}$.
Thus, if we simplify the right hand side of \eqref{eq:simplified left} we have
\begin{equation*}
\prod_{i=1}^kL_{i,i+1}\ordle
    \prod_{h_i^{(j,j+1)}\not\in\mathcal O}L_{i,i+1} \cdot
    \prod_{h_i^{(j,j+1)}\in\mathcal O}t_i h_i^{(j,j+1)}.
\end{equation*}
Divide both sides by $\prod_{h_i\not\in\mathcal O}L_{i,i+1}$ and we have\[
\prod_{h_i^{(j,j+1)}\in\mathcal O}L_{i,i+1}\ordle\prod_{h_i^{(j,j+1)}\in\mathcal O}t_i h_i^{(j,j+1)}.
\]
Recall that $\mathcal O$ was guaranteed to be nonempty,
so these products are greater than 1.
This contradicts the definition of an $S$-representation.

We have shown that the leading term of the elementary product of $\det A_k$
formed on the main diagonal is $\prod_{i=1}^k\sigma_{i+1,i}$,
while the leading terms of the remaining elementary products are strictly smaller.
The sum of the elementary products thus derives its leading term from the main diagonal,
whose leading term is the form described by (B).
\end{proof}

Finally we turn to the proof of Lemma \ref{lem: Chain implies certain lts}.

\begin{proof}
[Proof of Lemma \ref{lem: Chain implies certain lts}]Assume (A).
We must show (B).

For each $i=1,\ldots,m-1$ fix $\mathbf{h}^{\left(i,i+1\right)}$,
an $S$-re\-pre\-sen\-ta\-tion of $\SPoly{g_{i},g_{i+1}}$. We
have the system of $m-1$ equations\[
\begin{array}{rlllll}
Z_{1,2}g_{1} & +Z_{2,1}g_{2} & +h_{3}^{\left(1,2\right)}g_{3} & +\cdots & +h_{m}^{\left(1,2\right)}g_{m} & =0\\
 &  &  &  &  & \vdots\\
h_{1}^{\left(m-1,m\right)}g_{1} & +\cdots & +h_{m-2}^{\left(m-1,m\right)}g_{m-2} & +Z_{m-1,m}g_{m-1} & +Z_{m,m-1}g_{m} & =0.\end{array}\]
Eliminate $g_2$, \ldots, $g_{m-1}$ from the system.
By Lemmas~\ref{lem:remove some variables from system} and~\ref{lem:matrix nonsingular}
(with $x_i=g_{i+1}$ for $i=1,\ldots,m-2$)
we obtain $g_{1}P=g_{m}Q$ where\[
P=\left|\begin{array}{cccccc}
Z_{1,2} & Z_{2,1} & h_{3}^{\left(1,2\right)} & \cdots & h_{m-2}^{\left(1,2\right)} & h_{m-1}^{\left(1,2\right)}\\
h_{1}^{\left(2,3\right)} & Z_{2,3} & Z_{3,2} & \ddots & h_{m-2}^{\left(2,3\right)} & h_{m-1}^{\left(2,3\right)}\\
\vdots &  & \ddots\\
h_{1}^{\left(m-1,m\right)} & h_{2}^{\left(m-1,m\right)} & h_{3}^{\left(m-1,m\right)} &  & h_{m-2}^{\left(m-1,m\right)} & Z_{m-1,m}\end{array}\right|\]
and\[
Q=\left|\begin{array}{cccccc}
Z_{2,1} & h_{3}^{\left(1,2\right)} & \cdots & h_{m-2}^{\left(1,2\right)} & h_{m-1}^{\left(1,2\right)} & h_{m}^{\left(1,2\right)}\\
 & \ddots &  &  & \ddots & \vdots\\
h_{2}^{\left(m-2,m-1\right)} & h_{3}^{\left(m-2,m-1\right)} & \ddots & Z_{m-2,m-1} & Z_{m-1,m-2} & h_{m}^{\left(m-2,m-1\right)}\\
h_{2}^{\left(m-1,m\right)} & h_{3}^{\left(m-1,m\right)} & \cdots & h_{m-2}^{\left(m-1,m\right)} & Z_{m-1,m} & Z_{m,m-1}\end{array}\right|.\]
To show that $\lt{P}$ and $\lt{Q}$ have the form specified by the lemma,
apply an argument similar to the one used to prove Lemma~\ref{lem:matrix nonsingular}.
\end{proof}

Gr\"obner basis theory generalizes many algorithms for univariate polynomials
to systems of multivariate polynomials;
one oft-cited example is how Buchberger's algorithm to compute a Gr\"obner basis
can be viewed as a generalization of the Euclidean algorithm to compute the gcd.
We likewise expect relationships to exist between the $S$-poly\-nomials and the gcd's of polynomials.

Moreover, the construction of $S$-poly\-nomials relies on the computation of\[
    \sigma_{g_i,g_j}=\frac{\lcm\left(\lt{g_i},\lt{g_j}\right)}{\lt{g_i}}
\]which can be rewritten as\[
    \sigma_{g_i,g_j}=\frac{\lt{g_j}}{\gcd\left(\lt{g_i},\lt{g_j}\right)}.
\]Based on this, one might expect the existence of criteria on $S$-poly\-nomials
that relate the gcd of two polynomials with the gcd of their leading terms.

One such criterion exists for two polynomials: if $G=\{g_1,g_2\}$ is a Gr\"obner basis,
then the $S$-poly\-nomial of $g_1$ and $g_2$ reduces to zero,
and in addition $g_1=f_1p$ and $g_2=f_2p$ where $p=\gcd\left(g_1,g_2\right)$ and the leading terms of $f_1$ and $f_2$
are relatively prime \cite{Adams94}.
In this case, we infer a surprising fact. Observe that\begin{align*}
\lt{\gcd\left(g_1,g_2\right)}&=\lt{\gcd\left(f_1p,f_2p\right)}\\
    &=\lt{\gcd\left(f_1,f_2\right)\cdot p}.
\end{align*}Since $p$ is the gcd of $g_1$ and $g_2$,
we know that $f_1$ and $f_2$ must be relatively prime, so
\begin{align*}
\lt{\gcd\left(g_1,g_2\right)}&=\lt{1}\cdot\lt{p}\\
    &=\gcd\left(\lt{f_1},\lt{f_2}\right)\cdot\lt{p}\\
    &=\gcd\left(\lt{f_1}\lt{p},\lt{f_2}\lt{p}\right)\\
    &=\gcd\left(\lt{g_1},\lt{g_2}\right).
\end{align*}
Lemma~\ref{lem: Chain implies gcd commutes} generalizes this observation
in a way that does not require a Gr\"obner basis,
but does require the Extended Criterion!

\begin{lem}
\label{lem: Chain implies gcd commutes}Let $G\in\polyring^{m}$,
and suppose that the leading terms of $G$ satisfy the Extended Criterion.
Then (A)$\Rightarrow$(B) where
\begin{lyxlist}{(B)}
\item [{(A)}] Each of $\SPoly{g_{1},g_{2}}$, $\SPoly{g_{2},g_{3}}$, \ldots{},
$\SPoly{g_{m-1},g_{m}}$ has an $S$-re\-pre\-sen\-ta\-tion with
respect to $G$.
\item [{(B)}] $\gcd\left(\lt{g_{1}},\lt{g_{m}}\right)=\lt{\gcd\left(g_{1},g_{m}\right)}$.\closer
\end{lyxlist}
\end{lem}
\begin{proof}
Assume (A). We must show (B). For the sake of convenience, denote
$\lt{g_{i}}$ by $t_{i}$.

By Lemma~\ref{lem: Chain implies certain lts}, we have\[
g_{1}P=g_{m}Q\]
where\[
\lt{P}=\sigma_{g_{1},g_{2}}\sigma_{g_{2},g_{3}}\cdots\sigma_{g_{m-1},g_{m}}\quad\mbox{and}\quad\lt{Q}=\sigma_{g_{2},g_{1}}\sigma_{g_{3},g_{2}}\cdots\sigma_{g_{m},g_{m-1}}.\]
Let $p=\gcd\left(g_{1},g_{m}\right)$ and put $f_{1}=g_{1}/p$ and
$f_{m}=g_{m}/p$. Then\begin{equation}
f_{1}P=f_{m}Q.\label{eq: Pf_1 = Qf_4}\end{equation}
Since $f_{1},f_{m}$ are relatively prime, $f_{1}\mid Q$. Thus $\lt{f_{1}}$
divides $\lt{Q}$.

Observe that for any $i=1,\ldots,m-1$, we have\[
\sigma_{g_{i+1},g_{i}}=\frac{\lcm\left(t_{i},t_{i+1}\right)}{t_{i+1}}=\frac{t_{i}}{\gcd\left(t_{i},t_{i+1}\right)}.\]
Thus\[
\lt{f_{1}}\mid\frac{t_{1}t_{2}\cdots t_{m-1}}{\gcd\left(t_{1},t_{2}\right)\gcd\left(t_{2},t_{3}\right)\cdots\gcd\left(t_{m-1,m}\right)}.\]
Denote $\gcd\left(t_{i},t_{j}\right)$ by $d_{i,j}$. For all variables
$x$, we have\begin{align*}
\deg_{x}\lt{f_{1}} & \leq\deg_{x}\frac{t_{1}\cdots t_{m-1}}{d_{1,2}d_{2,3}\cdots d_{m-1,m}}.\end{align*}
Recall that $f_{1}=g_{1}/p$. For all variables $x$, we have \begin{align}
\deg_{x}t_{1}-\deg_{x}\lt{p} & \leq\sum_{1\leq i<m}\deg_{x}t_{i}-\sum_{1\leq i<m}\deg_{x}d_{i,i+1}\nonumber \\
\sum_{1\leq i<m}\deg_{x}d_{i,i+1} & \leq\deg_{x}\lt{p}+\sum_{1<i<m}\deg_{x}t_{i}.\label{eq: constraint on inner gcds}\end{align}

We claim that for all variables $x$, $\deg_{x}d_{1,m}\leq\deg_{x}\lt{p}$.
Let $x$ be arbitrary, but fixed. If $\deg_{x}t_{1}=0$ or $\deg_{x}t_{m}=0$,
the claim is trivially true. So assume $\deg_{x}t_{1}\neq0$ and $\deg_{x}t_{m}\neq0$.
We consider two cases.

If $\deg_{x}t_{1}\leq\deg_{x}t_{m}$, then $\deg_{x}d_{1,m}=\deg_{x}t_{1}$.
Recall that $t_{1}$, \ldots{}, $t_{m}$ satisfy EC. Therefore $\deg_{x}t_{1}\leq\deg_{x}t_{2}\leq\cdots\leq\deg_{x}t_{m}$.
Thus $\deg_{x}d_{i,i+1}=\deg_{x}t_{i}$ for all $i$ such that $1\leq i\leq m-1$.
Apply this to \eqref{eq: constraint on inner gcds} to obtain\[
\deg_{x}d_{1,m}=\deg_{x}t_{1}\leq\deg_{x}\lt{p}.\]

If $\deg_{x}t_{1}\geq\deg_{x}t_{m}$, a similar argument gives $\deg_{x}d_{1,m}\leq\deg_{x}\lt{p}$.

Since $x$ is arbitrary, $d_{1,m}$ divides $\lt{p}$, or equivalently
$\gcd\left(\lt{g_{1}},\lt{g_{m}}\right)$ divides $\lt{\gcd\left(g_{1},g_{m}\right)}$.
That $\lt{\gcd\left(g_{1},g_{m}\right)}$ divides $\gcd\left(\lt{g_{1}},\lt{g_{m}}\right)$
is trivial. Hence $\lt{\gcd\left(g_{1},g_{m}\right)}=\gcd\left(\lt{g_{1}},\lt{g_{m}}\right)$.
\end{proof}
The following result will be useful both for the proof of the Main
Theorem and for Section~\ref{sec: class of systems}.

\begin{cor}
\label{cor:EFC implies reprime lts}Let $G\in\polyring^{m}$, and
suppose that the leading terms of $G$ satisfy the Extended Criterion.
Then (A)$\Rightarrow$(B) where
\begin{lyxlist}{(B)}
\item [{(A)}] $\SPoly{g_{1},g_{2}}$, $\SPoly{g_{2},g_{3}}$, \ldots{},
$\SPoly{g_{m-1},g_{m}}$ all have $S$-re\-pre\-sen\-ta\-tions
with respect to $G$.
\item [{(B)}] If $p=\gcd\left(g_{1},g_{m}\right)$, then $\lt{g_{1}/p}$
and $\lt{g_{m}/p}$ are relatively prime.\closer
\end{lyxlist}
\end{cor}
\begin{proof}
Assume (A). Let $p=\gcd\left(g_{1},g_{m}\right)$, and denote $g_{1}/p$
and $g_{m}/p$ by $f_{1}$ and $f_{m}$, respectively. From Lemma~\ref{lem: Chain implies gcd commutes},
we know that\[
\gcd\left(\lt{g_{1}},\lt{g_{m}}\right)=\lt{p}.\]
Thus for any variable $x$,\begin{align*}
\deg_{x}\gcd\left(\lt{g_{1}},\lt{g_{m}}\right) & =\deg_{x}\lt{g_{1}}-\deg_{x}\lt{f_{1}}\\
 & =\deg_{x}\lt{g_{m}}-\deg_{x}\lt{f_{m}}.\end{align*}
Let $x$ be arbitrary, but fixed. If $\deg_{x}\lt{g_{1}}\leq\deg_{x}\lt{g_{m}}$,
then\[
\deg_{x}\lt{g_{1}}=\deg_{x}\gcd\left(\lt{g_{1}},\lt{g_{m}}\right)=\deg_{x}\lt{g_{1}}-\deg_{x}\lt{f_{1}},\]
so $\deg_{x}\lt{f_{1}}=0$. Similar reasoning shows that if $\deg_{x}\lt{g_{1}}\geq\deg_{x}\lt{g_{m}}$,
then $\deg_{x}\lt{f_{m}}=0$. It follows that $\lt{g_{1}/p}$ and
$\lt{g_{m}/p}$ are relatively prime.
\end{proof}
Theorem~\ref{thm: completion of sufficiency} is the main tool used
to prove the Main Theorem. Note that a similar statement holds for
Buchberger's lcm Criterion, although the chain needed for the lcm
Criterion, unlike the chain for the Extended Criterion, does not need
to use all the polynomials of $G$.

\begin{thm}
\label{thm: completion of sufficiency}Let $G\in\polyring^{m}$, and
suppose that the leading terms of $G$ satisfy the Extended Criterion.
Then (A)$\Rightarrow$(B) where
\begin{lyxlist}{(B)}
\item [{(A)}] $\SPoly{g_{1},g_{2}}$, $\SPoly{g_{2},g_{3}}$, \ldots{},
$\SPoly{g_{m-1},g_{m}}$ all have $S$-re\-pre\-sen\-ta\-tions
with respect to $G$.
\item [{(B)}] $\SPoly{g_{1},g_{m}}$ has an $S$-re\-pre\-sen\-ta\-tion
with respect to $G$.\closer
\end{lyxlist}
\end{thm}
\begin{proof}
Assume (A). We want to show (B). For the sake of convenience, denote
$\lt{g_{i}}$ by $t_{i}$.

Recall that\begin{equation}
\SPoly{g_{1},g_{m}}=\lc{g_{m}}\cdot\frac{\lcm\left(t_{1},t_{m}\right)}{t_{1}}\cdot g_{1}-\lc{g_{1}}\cdot\frac{\lcm\left(t_{1},t_{m}\right)}{t_{m}}\cdot g_{m}.\label{eq: preparing to move from f to c}\end{equation}
Let $p=\gcd\left(g_{1},g_{m}\right)$ where $\lc{p}=1$. Put $f_{1}=g_{1}/p$
and $f_{m}=g_{m}/p$. From Lemma~\ref{lem: Chain implies gcd commutes},
we know that $\gcd\left(\lt{g_{1}},\lt{g_{m}}\right)=\lt{\gcd\left(g_{1},g_{m}\right)}$.
This and the facts $\lc{f_{1}}=\lc{g_{1}}$ and $\lc{f_{m}}=\lc{g_{m}}$
give\begin{align*}
\lc{g_{1}}\cdot\frac{\lcm\left(t_{1},t_{m}\right)}{t_{m}} & =\lc{g_{1}}\cdot\frac{t_{1}t_{m}}{t_{m}\gcd\left(t_{1},t_{m}\right)}=\lc{f_{1}}\cdot\lt{f_{1}}.\\
 & \mbox{and}\\
\lc{g_{m}}\cdot\frac{\lcm\left(t_{1},t_{m}\right)}{t_{1}} & =\lc{g_{m}}\cdot\frac{t_{1}t_{m}}{t_{1}\gcd\left(t_{1},t_{m}\right)}=\lc{f_{m}}\cdot\lt{f_{m}}\end{align*}
This allows us to rewrite \eqref{eq: preparing to move from f to c}
as\begin{align*}
\SPoly{g_{1},g_{m}} & =\lc{f_{m}}\lt{f_{m}}\cdot g_{1}-\lc{f_{1}}\lt{f_{1}}\cdot g_{m}\\
 & =p\cdot\SPoly{f_{1},f_{m}}.\end{align*}
By Corollary~\ref{cor:EFC implies reprime lts}, the leading terms
of $f_{1}$ and $f_{m}$ are relatively prime; by Buchberger's gcd
Criterion, $\SPoly{f_{1},f_{m}}$ has an $S$-re\-pre\-sen\-ta\-tion
$\mathbf{h}$. It follows that $\mathbf{h}p=\left(h_{1}p,\ldots,h_{m}p\right)$
is an $S$-re\-pre\-sen\-ta\-tion of $\SPoly{g_{1},g_{m}}$.
\end{proof}
Theorem~\ref{thm: completion of sufficiency} provides us with sufficient
information to conclude that the Main Theorem is true. This may not
be clear, because we have discussed only $S$-re\-pre\-sen\-ta\-tions,
and not reduction to zero. To show how the two come together,
we need to recall two additional results. The first is the characterization
of Gr\"obner bases due to Lazard \cite{Lazard93}.

\begin{thm}
[Lazard's Characterization]\label{thm: Lazards characterization}Let
$G\in\polyring^{m}$. The following are equivalent.
\begin{lyxlist}{(B)}
\item [{(A)}] $G$ is a Gr\"obner basis with respect to $\ordering$.
\item [{(B)}] For every $i,j$ such that $1\leq i<j\leq m$, $\SPoly{g_{i},g_{j}}$
has an $S$-re\-pre\-sen\-ta\-tion with respect to $G$.\closer
\end{lyxlist}
\end{thm}

It turns out that Buchberger's characterization implies Lazard's,
thanks to the following Lemma \cite{BWK93}:

\begin{lem}
\label{lem: reduction implies representation}Let $G\in\polyring^{m}$
and let $i,j$ satisfy $1\leq i<j\leq m$. Then (A)$\Longrightarrow$(B)
where
\begin{lyxlist}{(B)}
\item [{(A)}] $\SPoly{g_{i},g_{j}}$ reduces to zero with respect to $G$.
\item [{(B)}] $\SPoly{g_{i},g_{j}}$ has an $S$-re\-pre\-sen\-ta\-tion
with respect to $G$.\closer
\end{lyxlist}
\end{lem}
\noindent However, the converse of Lemma~\ref{lem: reduction implies representation} is known to be false,
so the fact that Lazard's characterization implies Buchberger's is not obvious.
It depends on the fact that in Lazard's characterization, \emph{every} pair $(i,j)$
has an $S$-repre\-sen\-ta\-tion for $\SPoly{g_i,g_j}$,
whereas Lemma ~\ref{lem: reduction implies representation} deals only with \emph{one} $S$-repre\-sen\-ta\-tion.

We can now show how Theorem~\ref{thm: completion of sufficiency}
proves the Main Theorem.

\begin{proof}
[Proof of Main Theorem]That (A) implies (B) is trivial, so we assume
(B) and show (A). To prove (A), we will employ Lazard's Characterization.

From (B), every pair $\left(i,j\right)$ satisfies one of (B0)---(B3).
Let $i,j$ be such that $1\leq i<j\leq m$. Clearly $\SPoly{g_{i},g_{j}}$
has an $S$-re\-pre\-sen\-ta\-tion:
\begin{itemize}
\item if $\left(i,j\right)$ satisfies (B0), then by Lemma~\ref{lem: reduction implies representation};
\item if $\left(i,j\right)$ satisfies (B1) or (B2), then by well-known
results \cite{BWK93,Adams94,CLO97};
\item if $\left(i,j\right)$ satisfies (B3), then by Theorem~\ref{thm: completion of sufficiency}.
\end{itemize}
\noindent By Lazard's Characterization (Theorem \ref{thm: Lazards characterization}),
$G$ is a Gr\"obner basis with respect to $\ordering$.
\end{proof}

\section{\label{sec: class of systems}``Pham-like'' systems}

In this section, we describe a class of polynomial systems for which
the Extended Criterion provides a dramatic reduction in the number
of $S$-poly\-nomial computations required for verification (Corollary~\ref{cor: dramatic reduction}).

A well-studied system of polynomials is the \emph{Pham} system \cite[Chapter~6, p.~147]{Tapas}.

\begin{dfn}
[Pham system]Let $P\in\mathbb{F}\left[x_{1},x_{2},\ldots,x_{n}\right]^{n}$.
We say that $P$ is a \emph{Pham system} if $\lt{p_{i}}$ and $\lt{p_{j}}$
are relatively prime whenever $i\neq j$.\closer
\end{dfn}
Thanks to Theorem~\ref{thm:Buchberger's Criteria},
one can verify that any Pham system is a Gr\"obner basis without
checking any $S$-poly\-nomials at all. Now we obfuscate matters
somewhat through multiplication.

\begin{dfn}[Pham-like systems]\label{def:Pham-like system}
Suppose that $G=\left(g_{1},\ldots,g_{m}\right)$
has leading terms $\left(c_{1}d,\ldots,c_{m}d\right)$ where for all
$i=1,\ldots,m$,
\begin{itemize}
\item $c_{i}$ and $d$ are relatively prime, \emph{and}
\item for all $j\neq i$, $c_{i}$ and $c_{j}$ are relatively prime.
\end{itemize}
We call such $G$ a \emph{Pham-like system.\closer}
\end{dfn}
Consider the following question.

\begin{center}
\emph{Is a Pham-like system a Gr\"obner basis?}
\par\end{center}

\noindent The temptation may arise to answer in the affirmative, because
the cofactors of the leading terms' gcd are relatively prime,
which through some manipulation might allow Buchberger's gcd Criterion
to apply.
\emph{It does not.}
Numerous systems are not Gr\"obner bases even though this property
is true; for example,\begin{equation*}
g_{1}=xy+y,\quad g_{2}=xz.\label{eq: pham-like, not gb}\end{equation*}
So deciding whether $G$ is a Gr\"obner basis requires us to check
whether the $S$-poly\-nomials reduce to zero. We would like to avoid
checking all of them if possible.

To that end, we turn first to Buchberger's Criteria, but

\begin{itemize}
\item none of the leading terms $c_{i}d$, $c_{j}d$ are relatively prime;
and
\item for any pair $c_{i}d$ and $c_{j}d$, no $c_{k}d$ divides their lcm.
\end{itemize}
If we were to rely only on Buchberger's Criteria, we would have to
reduce all $m\left(m-1\right)/2$ $S$-poly\-nomials to zero to see
that a Pham-like system is a Gr\"obner basis.

However, the Extended Criterion allows us
to decide whether a Pham-like system is a Gr\"obner basis by checking
at most $m-1$ $S$-poly\-nomials,
\emph{even though Buchberger's Criteria provide no benefit.}

\begin{cor}
\label{cor: dramatic reduction}Let $G\in\polyring^{m}$ be a Pham-like
system. The following are equivalent:

(A) $G$ is a Gr\"obner basis with respect to $\ordlt$.

(B) The $S$-poly\-nomials $\SPoly{g_{1},g_{2}}$, $\SPoly{g_{2},g_{3}}$,
\ldots{}, $\SPoly{g_{m-1},g_{m}}$ reduce to zero with respect to
$G$.\closer
\end{cor}
\begin{proof}
That (A) implies (B) is trivial, so we assume (B) and show (A). From
(B), we know that $\SPoly{g_{1},g_{2}}$, $\SPoly{g_{2},g_{3}}$,
\ldots{}, and $\SPoly{g_{m-1},g_{m}}$ reduce to zero with respect
to $G$. It follows from Lemma~\ref{lem: reduction implies representation}
that they have $S$-re\-pre\-sen\-ta\-tions with respect to $G$.

For the sake of convenience, denote $\lt{g_{i}}$ by $t_{i}$.
Write $t_i = c_i d$ where $c_i$ and $d$ are as in
Definition~\ref{def:Pham-like system}.
Recall that $\gcd\left(c_{i},t_{j}\right)=1$ whenever $i\neq j$;
inspection shows that the list of terms $\left(t_{1},t_{2},\ldots,t_{m}\right)$
satisfies the Extended Criterion. By Theorem~\ref{thm: completion of sufficiency},
$\SPoly{g_{1},g_{m}}$ has an $S$-re\-pre\-sen\-ta\-tion with
respect to $G$. Let $p_{1,m}=\gcd\left(g_{1},g_{m}\right)$ and choose
$f_{1},f_{m}\in\polyring$ such that
\begin{itemize}
\item $g_{1}=f_{1}p_{1,m}$, and
\item $g_{m}=f_{m}p_{1,m}$.
\end{itemize}
\noindent Recall Lemma \ref{lem: Chain implies gcd commutes} and
the assumption that $c_{1}$ is relatively prime to $t_{m}$; then\[
d=\gcd\left(c_{1}d,c_{m}d\right)=\gcd\left(\lt{g_{1}},\lt{g_{m}}\right)=\lt{p_{1,m}}.\]
Thus\[
c_{1}d=t_{1}=\lt{g_{1}}=\lt{f_{1}p_{1,m}}=\lt{f_{1}}\lt{p_{1,m}}=\lt{f_{1}}d,\]
whence $c_{1}=\lt{f_{1}}$. Similarly, $c_{m}=\lt{f_{m}}$.

Inspection shows that the list of terms $\left(t_{1},t_{m},t_{m-1},\ldots,t_{3},t_{2}\right)$
also satisfies the Extended Criterion. We now know that $\SPoly{g_{1},g_{m}}$
has an $S$-re\-pre\-sen\-ta\-tion with respect to $G$, so we
can reason as before that there exist $\varphi_{1},\varphi_{2},p_{1,2}\in\polyring$
such that
\begin{itemize}
\item $g_{1}=\varphi_{1}p_{1,2}$,
\item $g_{2}=\varphi_{2}p_{1,2}$,
\item $p_{1,2}=\gcd\left(g_{1},g_{2}\right)$, and
\item the leading terms of $\varphi_{1}$ and $\varphi_{2}$ are relatively
prime.
\end{itemize}
\noindent As before, we obtain $d=\lt{p_{1,2}}$ and $c_{1}=\lt{\varphi_{1}}$.
Thus $\lt{f_{1}}=\lt{\varphi_{1}}$. We claim that in fact $f_{1}=\varphi_{1}$.
By way of contradiction, assume that $f_{1}$ and $\varphi_{1}$ are
not equal. From $f_{1}p_{1,m}=\varphi_{1}p_{1,2}$ we conclude that
$f_{1}$ has a common factor with $p_{1,2}$ or $\varphi_{1}$ has
a common factor with $p_{1,m}$---but this contradicts the hypothesis
that $c_{1}$ is relatively prime to $d$. Hence $f_{1}=\varphi_{1}$
and $p_{1,m}=p_{1,2}$. Write $p=p_{1,m}$, $g_{1}=f_{1}p$, $g_{2}=f_{2}p$,
and $g_{m}=f_{m}p$.

Proceeding in like fashion, we can factor every $g_{i}$ as $g_{i}=f_{i}p$
such that $\lt{f_{i}}$ and $\lt{f_{j}}$ are relatively prime whenever
$i\neq j$. By Theorem~\ref{thm:Buchberger's Criteria}, $F=\left(f_{1},f_{2},\ldots,f_{m}\right)$
is a Gr\"obner basis with respect to $\ordlt$. Let $i,j$ be arbitrary,
but fixed. Assume $1\leq i<j\leq m$. By Lazard's Characterization,
$\SPoly{f_{i},f_{j}}$ has an $S$-re\-pre-sen\-ta\-tion $\mathbf{h}^{\left(i,j\right)}$.
This implies that $\SPoly{g_{i},g_{j}}$ has an $S$-re\-pre\-sen\-ta\-tion
$p\mathbf{h}^{\left(i,j\right)}=\left(ph_{1}^{\left(i,j\right)},\ldots,ph_{m}^{\left(i,j\right)}\right)$.
Since $i$ and $j$ are arbitrary, by Lazard's Characterization $G$
is a Gr\"obner basis with respect to $\ordlt$.
\end{proof}

\section{Acknowledgment}

The author would like to thank the anonymous referee whose comments greatly improved
the quality of the paper.

\bibliographystyle{jcm}
\bibliography{/home/perry/common/Research/researchbibliography}

\end{document}